\documentclass[12pt]{article}
\usepackage{amsmath,amssymb,amsthm,graphicx, color, caption}
\usepackage{hyperref}
\usepackage{float}
\usepackage{fullpage}

\newtheorem{theorem}{Theorem}
\newtheorem{lemma}[theorem]{Lemma}
\newtheorem{observation}[theorem]{Observation}
\newtheorem{corollary}[theorem]{Corollary}

\newtheorem{question}[theorem]{Question}

\newtheorem{prop}[theorem]{Proposition}

\newenvironment{1claim}[1]{\par\noindent\underline{Claim 1:}\space#1}{}
\newenvironment{claimproof}[1]{\par\noindent\underline{Proof:}\space#1}{}
\newenvironment{1claim2}[1]{\par\noindent\underline{Claim 2:}\space#1}{}
\newenvironment{1claim3}[1]{\par\noindent\underline{Claim 3:}\space#1}{}
\newenvironment{1claim4}[1]{\par\noindent\underline{Claim 4:}\space#1}{}
\newenvironment{2claim}[1]{\par\noindent\underline{Claim 1:}\space#1}{}
\newenvironment{2claim2}[1]{\par\noindent\underline{Claim 2:}\space#1}{}
\newenvironment{2claim3}[1]{\par\noindent\underline{Claim 3:}\space#1}{}
\newenvironment{2claim4}[1]{\par\noindent\underline{Claim 4:}\space#1}{}
\newenvironment{2claim5}[1]{\par\noindent\underline{Claim 5:}\space#1}{}
\newenvironment{2claim6}[1]{\par\noindent\underline{Claim 6:}\space#1}{}
\newenvironment{2claim7}[1]{\par\noindent\underline{Claim 7:}\space#1}{}
\newenvironment{2claim8}[1]{\par\noindent\underline{Claim 8:}\space#1}{}

\begin{document}

\title{Reflexive coloring complexes for 3-edge-colorings of cubic graphs}

\author{
Fiachra Knox\thanks{Supported in part by a PIMS Postdoctoral Fellowship} \hspace{5mm} Bojan Mohar\thanks{Supported in part by the NSERC Discovery Grant R611450 (Canada), by the Canada Research Chairs program,
and by the Research Project J1-8130 of ARRS (Slovenia).}
\thanks{On leave from IMFM, Department of Mathematics, University
of Ljubljana.} \hspace{5mm} Nathan Singer \\[5mm]
Department of Mathematics\\ Simon Fraser University\\
Burnaby, BC V5A 1S6, Canada
}

\date{\today}

\maketitle

\begin{abstract}
Given a 3-colorable graph $X$, the 3-coloring complex $B(X)$ is the graph whose vertices are all the independent sets which occur as color classes in some 3-coloring of $X$.  Two color classes $C,D \in V(B(X))$ are joined by an edge if $C$ and $D$ appear together in a 3-coloring of $X$.  The graph $B(X)$ is 3-colorable.  Graphs for which $B(B(X))$ is isomorphic to $X$ are termed reflexive graphs.  In this paper, we consider 3-edge-colorings of cubic graphs for which we allow half-edges. Then we consider the 3-coloring complexes of their line graphs. The main result of this paper is a surprising outcome that the line graph of any connected cubic triangle-free outerplanar graph is reflexive. We also exhibit some other interesting classes of reflexive line graphs.
\end{abstract}

\section{Introduction}

In \cite{TutteNotes}, Tutte examined how many connected components the 4-coloring complex (see the definition below) of a triangulation of the plane could have. This was a question he had examined since the times of the Four Color Conjecture \cite{Tutte1}.

Edge-colorings of cubic graphs appear naturally in this setting. By the well-known coloring-flow duality, 4-colorings of triangulations are in bijective correspondence with 3-edge-colorings of their dual cubic graphs. Nevertheless, the coloring complex corresponding to 4-colorings of a triangulation and the coloring complex of 3-edge-colorings of the dual cubic graph may be very different. For example, the coloring complex of the icosahedron is connected, while all ten different 3-edge-colorings of the dual dodecahedron give separate components in the corresponding coloring complex.

Early on, this kind of questions raised attention, and Biggs \cite{Biggs} showed that the Coxeter graph, which has 28 vertices, has twice as many 3-edge-colorings, and the corresponding coloring complex is the line graph of a 2-arc-transitive cubic graph $B$ of order 56. From the list of known 2-arc-transitive graphs, Biggs concluded that this graph $B$ is the famous Klein map of genus 3 and of type $\{7,3\}_8$. It was only later that Fisk \cite{Fisk1} actually discovered that there was an error in this conclusion, because the coloring complex in question is disconnected, and is in fact isomorphic to two copies of the line graph of the Coxeter graph (which by itself is an interesting phenomenon).

In his attempts to answer questions that were posed by Tutte and others, Fisk \cite{Fisk1} introduced the notion of a reflexive coloring complex (see Section \ref{sec:2} for definition). Furthermore, he established in \cite{Fisk1} that the 3-coloring complexes of line graphs of  ``cubic'' cycles and paths are reflexive. By cubic cycles and paths we mean cubic graphs obtained from these by adding half-edges. In his monograph  \cite{Fiskbook}, Fisk provided a number of further developments in this area.

This paper continues the work started by Fisk. Our main result shows that the line graph of a connected cubic outerplanar graph $G$ (with added half-edges) is reflexive if and only if $G$ is triangle-free. This result, which appears as Theorem \ref{thm: Main}, gives a large infinite class of reflexive line graphs. In the second part of the paper we discover several infinite families of additional, interesting examples of reflexive line graphs.

Our results may be just the tip of an iceberg. The real question that remains open is why there are so many reflexive graphs (and why there are any at all). Our results and extensive computational evidence say that graphs with a large number of colorings tend to be reflexive, although this appears counterintuitive from another perspective: when we have many colorings, the coloring complex is large, so it may have too many colorings for the original graph to be reflexive.

\section{The coloring complex}
\label{sec:2}

Let $X$ be a $k$-colorable graph and let us consider a \emph{$k$-coloring} of $X$, which we treat as a partition
$\{V_1,V_2,\dots, V_k\}$ of $V(X)$ into $k$ independent sets $V_i$ ($1\le i\le k$). Here, some of the parts of the partition may be empty, but in all considered cases we will have the property that $X$ contains a $(k-1)$-clique, in which case at most one of the parts will be the empty set. The independent sets $V_i$ arising in a $k$-coloring of $X$ are called the \emph{color classes} of the coloring. The \emph{k-coloring complex}\footnote{It was Tutte \cite{Tutte1} who used the word ``complex'' for $B(X)$. The reason for this terminology is that $B(X)$ can be viewed as a simplicial complex in which each $k$-clique corresponding to a $k$-coloring of $X$ is made into a simplex of dimension $k-1$.} $B(X)$ is the graph whose vertices are the color classes of all $k$-colorings of $X$. Two vertices $C,D \in V(B(X))$ are joined by an edge if the color classes $C$ and $D$ appear together in a $k$-coloring of $X$.

It is not hard to see that the graph $B(X)$ is $k$-colorable under the following simple condition.

\begin{lemma}[Fisk \cite{Fisk1}]\label{lem:3Col}
Let $X$ be a $k$-colorable graph that contains a $(k-1)$-clique.  Then $B(X)$ is $k$-colorable.
\end{lemma}

\begin{proof}
Let $Q$ be a $(k-1)$-clique in $X$. For $i\in [k-1]$, let $\mathcal{C}_i$ be the set of all the color classes in $k$-colorings of $X$ which contain the $i$th vertex of $Q$.
Then $\{ \mathcal{C}_i : i \in [k-1] \} \cup \{ (V(B(X)) \setminus \cup_{i=1}^{k-1} \mathcal{C}_i)\}$ is a $k$-coloring of $B(X)$.
\end{proof}

When $B(X)$ is $k$-colorable, we can consider its $k$-coloring complex $B^{2}(X) = B(B(X))$.
It appears surprising that $X$ and $B^2(X)$ are closely related.

\begin{lemma}[Fisk \cite{Fisk1}]\label{lem:phi_X}
Let $X$ be a $k$-colorable graph without isolated vertices in which each edge is contained in a $(k-1)$-clique. Then the mapping
$$
    \phi_X :  v \mapsto \{C \in V(B(X)) \mid x \in C \}\quad (v\in V(X))
$$
is a graph homomorphism $X \to B^2(X)$.
\end{lemma}

The graph homomorphism outlined in Lemma \ref{lem:phi_X},
$$\phi_{X}: X \to B^2(X)$$
maps a vertex $v\in V(X)$ to the set of all color classes of $X$ containing $v$. We will refer to it as a \emph{canonical homomorphism}. Graphs for which the canonical homomorphism is an isomorphism are termed \emph{reflexive graphs} and they will be our main concern in this paper. There are no obvious reasons why any graph would be reflexive, yet there are interesting infinite classes. Our main goal is trying to understand this phenomenon.

The following is a necessary condition for $\phi_X$ to be reflexive. We say that the graph $X$ is \emph{colorful} (for $k$-colorings) if for any two vertices $x,y$ there exists a $k$-coloring, which has $x$ and $y$ in different color classes.
Let us add the following basic observation:

\begin{observation}
\label{obs:colorful}
A graph $X$ is colorful for $k$-colorings if and only if the mapping $\phi_X$ is injective.
\end{observation}

Let us now assume that $k=3$. (This assumption will be kept throughout the paper.) The 3-coloring complex $B(X)$ of a graph $X$ is composed of triangles, one triangle for every 3-coloring of $X$. The following result shows that $B(X)$ has no other triangles and also shows that the triangles are edge-disjoint.

\begin{lemma} \label{lem: TriFree}
Let $X$ be a 3-chromatic graph without isolated vertices. Then any triangle in $B(X)$ represents a 3-coloring of $X$.
Consequently, each edge of $B(X)$ is contained in precisely one triangle.
\end{lemma}

\begin{proof}
Suppose, for a contradiction, that $B(X)$ contains a triangle $C_1C_2C_3$ which does not represent a 3-coloring of $X$.
Each edge of $C_1C_2C_3$ must be in a 3-coloring of $X$.
Let $C_{ij}$ be the third color class of the 3-coloring containing $C_{i}$ and $C_{j}$, for each $ i \neq j $ in $ \{1,2,3\}$.

Now, \begin{math} C_1 \cap C_2 = C_2 \cap C_3 = C_3 \cap C_1 = \emptyset \end{math}.
Hence, $C_1 \subseteq C_{23}$, $C_2 \subseteq C_{31}$ and \begin{math} C_3 \subseteq C_{12} \end{math}.
Let \begin{math} H = V(X) \setminus (C_1 \cup C_2 \cup C_3) \end{math}.
Then $ C_{12} = V(X) \setminus (C_1 \cup C_2 ) = C_3 \cup H$, $C_{23} = C_1 \cup H$ and \begin{math} C_{31} = C_2 \cup H \end{math}.
If \begin{math} H = \emptyset \end{math}, then $C_{12} = C_3$ and $C_1C_2C_3$ is a 3-coloring of $X$, which is a contradiction. Thus, \begin{math} H \neq \emptyset \end{math}; let us take a vertex $v\in H$. Since $X$ has no isolated vertices and $H$ is an independent set in $X$, $v$ has a neighbour \begin{math} u\notin H \end{math}.
Without loss of generality, suppose that \begin{math} u \in C_1 \end{math}.
Then, as \begin{math} C_1 \subseteq C_{23} \end{math}, \begin{math} u,v \in C_{23} \end{math}.
However, this is a contradiction, since $u$ and $v$ are neighbours, while \begin{math} C_{23} \end{math} is an independent set.

Let us now consider an edge $AC$ in $B(X)$. By the above, any triangle containing $AC$ corresponds to a 3-coloring with color classes $A$ and $C$. But the third color class is just the complement of $A\cup C$, so there is only one such 3-coloring.
\end{proof}

As we will primarily concern ourselves with 3-edge-colorings in this article, we add a few related definitions.
Firstly, given a 3-edge-colorable graph $G$, we define the \emph{3-edge-coloring complex} of $G$ as the 3-coloring complex of the line graph $L(G)$ of $G$.
Furthermore, we say that a 3-edge-colorable graph $G$ is \emph{edge-reflexive} if $L(G)$ is a reflexive graph, and that $G$ is \emph{edge-colorful} if $L(G)$ is colorful.

When we refer to a \emph{cubic graph}, we allow \emph{half-edges}.
These are edges which are incident with only one vertex.
In this way, we can treat every graph of maximum degree three as a cubic graph by adding half-edges to the vertices of smaller degrees. (On the other hand, we do not allow double edges, since cubic graphs containing double edges cannot be edge-colorful, unless we have a triple edge.) With this understanding, we will in particular speak of \emph{cubic paths}, \emph{cubic cycles} and \emph{cubic trees}. We refer to Figure \ref{fig:cubictree} for some examples.

\begin{figure}
\centering
\includegraphics[width=0.56\textwidth]{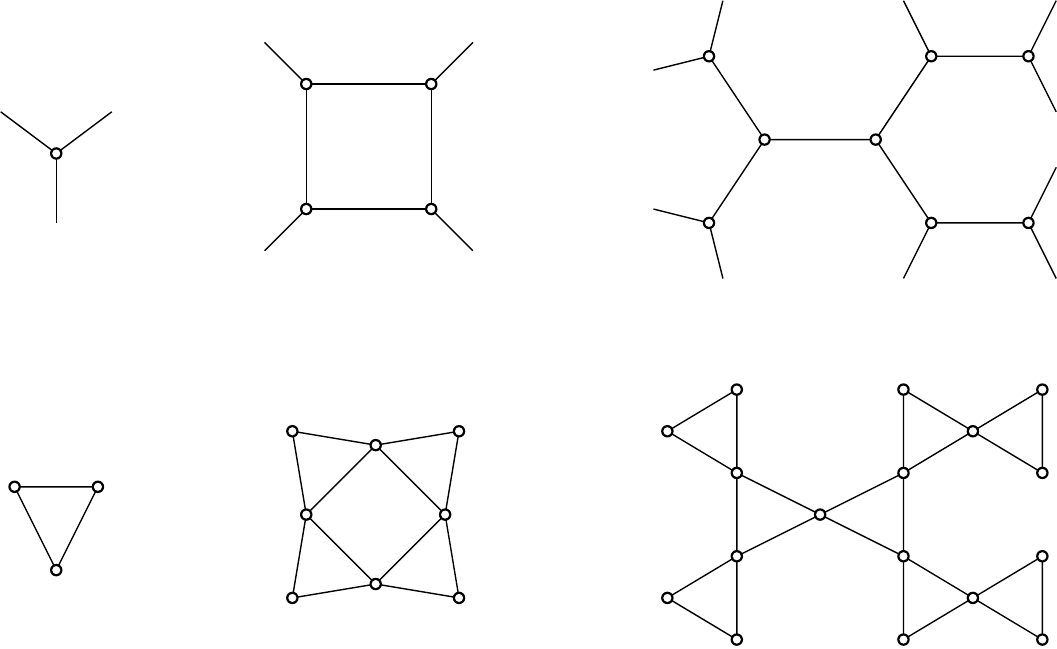}
\caption{Cubic vertex, cubic 4-cycle and a cubic tree with their line graphs.}
  \label{fig:cubictree}
\end{figure}

The following observation (combined with Observation \ref{obs:colorful}) shows that cubic graphs that contain triangles cannot be edge-reflexive.

\begin{observation}
\label{obs:triangle-free}
A cubic graph containing a triangle is not edge-colorful.
\end{observation}

\begin{proof}
Let $abc$ be a triangle in $G$ and let $e$ be the third edge incident with the vertex $a$. Then the edges $e$ and $bc$ have the same color in every 3-edge coloring of $G$.
\end{proof}

We also have the following lemma.

\begin{lemma}
\label{lem:edge-reflexive-countC3}
Let $G$ be a cubic edge-colorful graph of order $n$ and let $X=L(G)$. The following statements are equivalent:
\begin{enumerate}
  \item [\rm (i)] $G$ is edge-reflexive.
  \item [\rm (ii)] $B(X)$ has precisely $n$ $3$-colorings.
  \item [\rm (iii)] For every $3$-coloring $\{{\mathcal A,\mathcal B,\mathcal C}\}$ of $B(X)$, there is a vertex in $G$ with incident edges $e,f,g$ such that $\{{\mathcal A,\mathcal B,\mathcal C}\} = \{\phi_X(e),\phi_X(f),\phi_X(g)\}$.
\end{enumerate}
\end{lemma}

\begin{proof}
As argued in the proof of Lemma \ref{lem:3Col}, all partitions of $V(B(X))$ of the form $\{\phi_X(e),\allowbreak \phi_X(f),\phi_X(g)\}$ are 3-colorings of $B(X)$. This shows that (ii) and (iii) are equivalent.

Next, observe that $\phi_X: X\to B^2(X)$ is injective and that every triangle in $X$ corresponds to a vertex of $G$. Thus $X$ has precisely $n$ triangles. This yields equivalence of (i) and (ii).
\end{proof}

Finally, let us add a necessary condition for reflexivity.

\begin{lemma}
\label{lem:necessary degrees 2^d}
Let $X$ be a graph that is reflexive for 3-colorings. If $v$ is a vertex of degree $d$ in $X$, then $d=2^t$, where $t$ is the number of components of the bipartite graph $B(X)-\phi_X(v)$.
\end{lemma}

\begin{proof}
The set $\phi_X(v)$ is the color class for 3-colorings of $B(X)$, and since $B(B(X))$ is isomorphic to $X$, the number of 3-colorings of $B(X)$ that have $\phi_X(v)$ as one of the color classes is equal to $d/2$ (each coloring contributes 2 towards the degree by Lemma \ref{lem: TriFree}). Thus, $d/2$ is equal to the number of 2-colorings of $B(X)-\phi_X(v)$. In particular, this subgraph is bipartite, and it is clear that the number of 2-colorings is equal to $2^{t-1}$. Thus, $d = 2\cdot 2^{t-1} = 2^t$.
\end{proof}

When applied to edge-colorings, we obtain the following corollary.

\begin{corollary}
\label{cor:necessary 1/2 components}
Let $G$ be a cubic edge-reflexive graph and let $X=L(G)$. For each edge $e$ of $G$, the graph $B(X)-\phi_X(e)$ is bipartite, and its either connected (when $e$ is a half-edge), or has precisely two connected components (when $e$ is a full edge).
\end{corollary}

\section{Outerplanar graphs}

\subsection{Overview}

Throughout the remainder of this article, we will assume that cubic graphs $G$ satisfy the hypotheses of Theorem~\ref{thm: Main} below: they are always connected, and they may have half-edges.
Recall that a graph $G$ is \emph{outerplanar} if $G$ has a planar drawing in which all of its vertices appear on the unbounded face of the drawing. We will consider cubic triangle-free outerplanar graphs. Let us observe that all such graphs must have at least four half-edges.

Fiorini \cite{FIORINI} proved that every cubic outerplanar graph is 3-edge-colorable. The same proof shows that these graphs are edge-colorful whenever they are triangle-free.

The main result of this paper is the following somewhat surprising theorem, whose proof occupies the rest of this section.

\begin{theorem} \label{thm: Main}
Let $G$ be a connected cubic outerplanar graph.  Then $G$ is edge-reflexive if and only if it is triangle-free.
\end{theorem}

The nontrivial direction of Theorem \ref{thm: Main} will be proved in three steps.
First, we will reduce the problem to 2-connected graphs.
Then we will exhibit two operations which are used to construct all 2-connected cubic triangle-free outerplanar graphs from a 4-cycle.
Finally, we will show that edge-reflexivity is preserved under both of these operations, completing our argument that any connected, triangle-free, outerplanar graph $G$ is edge-reflexive.

Before proceeding, we will need to establish a brief lemma, for which we need to introduce a definition.
Suppose that $S\subseteq V(X)$ and that we have a 3-coloring $\Gamma = \{A,B,C\}$ of $X[S]$. A coloring $\{A',B',C'\}$ of $X$ is an \emph{extension} of the coloring $\Gamma$ if $A\subseteq A'$, $B\subseteq B'$, and $C\subseteq C'$ (or some permutation of color classes satisfies the same inclusion relations). If $v$ is a vertex and for every extension of $\Gamma$, $v$ is in the color class containing the same color class in $\Gamma$, then we say that \emph{the color of $v$ is determined} by $\Gamma$ (or that the coloring $\Gamma$ \emph{determines the color} of $v$). If colors of all vertices are determined by $\Gamma$ and $\Gamma$ has an extension, we say that $\Gamma$ \emph{determines} the extension (or that $\Gamma$ \emph{extends uniquely} to $X$).

\begin{lemma} \label{lem:Determined}
Let $F$ be a graph with a 3-coloring whose color classes are $A$, $B$ and $C$.  Suppose that $F$ has no isolated vertices, that every edge of $F$ is contained in exactly one triangle, and that $F[B \cup C]$ is connected.  Then any $3$-coloring of $F$ that uses at least two colors on $A$ is determined by its restriction to $A$.
\end{lemma}

\begin{proof}
Let $x_0$ and $y$ be two vertices of $A$ which have different colors, and let $v, w \in B$ be neighbours of $x_{0}$ and $y$, respectively.
Let $P = v_1 v_2 \ldots v_n$ be a path in $F[B \cup C]$ with $v_1 = v$ and $v_n = w$.
For each $i \in [n-1]$, let $x_i \in A$ be the common neighbour of $v_i$ and $v_{i+1}$.
Relabel $y$ as $x_n$.
Let $j \in [n]$ be such that $x_j$ and $x_{j-1}$ have different colors.
Such an index exists because $x_{0}$ and $x_{n}$ have different colors.
Then $v_j$ is adjacent to both of these vertices, and hence its color is determined.
Relabel $v_j$ as $z$.

Now let $u$ be any vertex in $B \cup C$.
Let $Q = u_1 u_2 \ldots u_m$ be a path in $F[B \cup C]$ with $z = u_1$ and $u = u_m$.
The color of $u_1$ is determined, and whenever the color of $u_i$ is determined, so is the color of $u_{i+1}$
(since it is in a triangle with $u_i$ and a vertex of $A$, both of whose colors are determined).
By induction the color of $u = u_m$ is determined.
Since $u$ was arbitrary, the entire coloring is determined.
\end{proof}

\subsection{Cutedges and reflexivity}

In this section, we will prove that a cubic graph $G$ is edge-reflexive whenever each block of $G$ is edge-reflexive.  Thus, we will be able to reduce edge-reflexivity questions to 2-edge-connected graphs.

Let $G$ be a cubic graph with a cutedge $e$ joining vertices $a$ and $b$. By \emph{cutting the edge} $e$ we obtain two cubic graphs $H$ and $K$ which are obtained from $G-e$ by adding a half-edge to $a$ and $b$, respectively. The added half-edge will be considered to be the same as the removed edge $e$, so that we can consider $E(H)\subseteq E(G)$, $E(K)\subseteq E(G)$, and $E(H)\cap E(K) = \{e\}$.

\begin{lemma}\label{lem:gluinghalfedges}
Let $G$ be a cubic graph with a cutedge $e$. Let $H$ and $K$ be cubic graphs obtained from $G$ by cutting the edge $e$.
If $H$ and $K$ are edge-reflexive, then $G$ is edge-reflexive, too.
\end{lemma}

\begin{proof}
We let $X=L(H)$, $Y=L(K)$, and $X'=L(G)$.
Let $w$ and $x$ be the neighbours of $e$ in $X$, and let $y$ and $z$ be the neighbours of $e$ in $Y$.
Before proceeding further, let us first observe that since $X$ and $Y$ are reflexive, they are also colorful. This implies that $X'$ is also colorful.  Thus, $\phi_{X'}$ is an injective homomorphism $X' \rightarrow B^{2}(X')$, and therefore it suffices to show that $X'$ and $B^{2}(X')$ have the same number of triangles. Triangles in $X'$ correspond to vertices in $V(G) = V(H)\cup V(K)$, while the triangles in $B^2(X')$ correspond to 3-colorings of $B(X')$.

Now, for any color class $C \in V(B(X))$, let $\mathcal{F}_C$ be the set of color classes of $X'$ which coincide with $C$ on the vertices of $X$. If $e\in C$, we also write $\mathcal{F}_C^e$ to denote the same set of color classes.
Observe that when $e \in C$ and $D \in \mathcal{F}_C$, we have $y, z \notin D$.
On the other hand, if $e \notin C$ then exactly one of $y$ and $z$ is in $D$.
In this case, we partition $\mathcal{F}_C$ into the subset $\mathcal{F}_{C}^{y}$ consisting of color classes which contain $y$, and the subset $\mathcal{F}_{C}^{z}$ of color classes containing $z$.
We refer to the set $\mathcal{F}_{C}^{e}$, or the sets $\mathcal{F}_{C}^{y}$ and $\mathcal{F}_{C}^{z}$ as the \emph{clusters} of $B(X')$ corresponding to $C$.
Each cluster is an independent set in $B(X')$.
We say that two clusters corresponding to $C, C' \in V(B(X))$ are \emph{adjacent} if $CC' \in E(B(X))$.

\medskip
\begin{1claim}
For any triangle $ACD$ in $B(X)$ with $e \in A$, the induced subgraph of $B(X')$ on $\mathcal{F}_{A}^{e} \cup \mathcal{F}_{C}^{y} \cup \mathcal{F}_{D}^{z}$ is isomorphic to $B(Y)$.  The same is true for $\mathcal{F}_{A}^{e} \cup \mathcal{F}_{C}^{z} \cup \mathcal{F}_{D}^{y}$.
\end{1claim}
\medskip
\begin{claimproof}
To prove the claim, consider the map $\psi : \mathcal{F}_{A}^{e} \cup \mathcal{F}_{C}^{y} \cup \mathcal{F}_{D}^{z} \rightarrow B(Y)$ which takes a color class $B$ of $X'$ to the intersection of $B$ with $V(Y)$.
Since the restriction of $B$ to $V(X)$ is identical to precisely one of $A, C$ or $D$ it is easy to see that $\psi$ is injective.
Further, we can form a color class of $X'$ from any color class of $Y$ by combining it with one of $A$, $C$ and $D$, provided that they agree on containment of $v$;
hence $\psi$ is surjective.

It remains to show that, for $B,B' \in \mathcal{F}_{A}^{e} \cup \mathcal{F}_{C}^{y} \cup \mathcal{F}_{D}^{z}$, $B$ is adjacent to $B'$ if and only if $\psi(B)$ is adjacent to $\psi(B')$.
Suppose that $B \in \mathcal{F}_{A}^{e}$ and $B' \in \mathcal{F}_{C}^{y}$.
If $B$ and $B'$ are adjacent in $B(X')$ then there is a color class $F \in B(X')$ such that $BB'F$ is a triangle in $B(X')$.
Note that $F \in \mathcal{F}_{D}^{z}$.
The image of this triangle under $\psi$ is also a triangle, and hence $\psi(B)$ is adjacent to $\psi(B')$.
On the other hand, if $\psi(B)$ and $\psi(B')$ are adjacent in $B(Y)$, then there is a triangle $\psi(B) \psi(B') J \subseteq B(Y)$ since all edges in $B(Y)$ come from 3-colorings of $Y$.
Now $B, B'$ and $D \cup J$ form a coloring of $X'$, and hence $B$ and $B'$ are adjacent in $B(X')$.
The proofs for the cases when $B \in \mathcal{F}_{A}^{e}, B' \in \mathcal{F}_{D}^{z}$ and $B \in \mathcal{F}_{C}^{y}, B' \in \mathcal{F}_{D}^{z}$ are similar.
This proves the claim.
\end{claimproof}
\medskip

\begin{1claim2}
The subgraph of $B(X')$ on color classes in $\mathcal{F}_{D}^{z} \cup \mathcal{F}_{C}^{y}$ is connected and bipartite, and is isomorphic to $B(Y) - \phi_Y(e)$.
Further, the graph $B^{*}$ obtained from $B(Y)$ by deleting the edges of $B(Y) - \phi_Y(e)$ is connected and isomorphic to each of the bipartite graphs between $\mathcal{F}_{A}^{e}$ and $\mathcal{F}_C$, as well as between $\mathcal{F}_{A}^{e}$ and $\mathcal{F}_{C}^{y} \cup \mathcal{F}_{D}^{z}$.
\end{1claim2}
\medskip
\begin{claimproof}
$B(Y) - \phi_Y(e)$ is connected by Corollary \ref{cor:necessary 1/2 components}.
Furthermore, every edge of $B(Y) - \phi_Y(e)$ is contained in a triangle in $B(Y)$ using no other edge of $B(Y) - \phi_Y(e)$.
Therefore, $B^{*}$ is also connected since the edges of such triangles can be used in $B^*$ to replace any of the removed the edges.
Now, suppose that $ACD$ is a triangle in $B(X)$, where $e \in A$.
Then the bipartite graph between $\mathcal{F}_{D}^{z}$ and $\mathcal{F}_{C}^{y}$ is isomorphic to $B(Y) - \phi_Y(e)$,
while the bipartite graphs between $\mathcal{F}_{A}^{e}$ and $\mathcal{F}_C$ and between $\mathcal{F}_{A}^{e}$ and $\mathcal{F}_{C}^{y}\cup \mathcal{F}_{D}^{z}$ are each isomorphic to $B^{*}$, in the latter case, by Claim 1.
Hence both of these are also connected, establishing the claim.
\end{claimproof}
\medskip

We say that a coloring is \emph{constant} on a set $S$ of vertices if $S$ is contained in a color class of the coloring.
If a coloring $\chi$ of $B(X')$ is constant on each of $\mathcal{F}_{C}^{y}$ and $\mathcal{F}_{C}^{z}$, we say that $\chi$ is \emph{near-constant} on $\mathcal{F}_C$.
\medskip

\begin{1claim3}
Let $ACD$ be a triangle in $B(X)$, where $e \in A$,
and let $\chi$ be a coloring of $B(X')$ that is constant on one of $\mathcal{F}_C$ or $\mathcal{F}_D$.
Then $\chi$ is constant on each of $\mathcal{F}_{A}^{e}$, $\mathcal{F}_C$ and $\mathcal{F}_D$.
\end{1claim3}
\medskip
\begin{claimproof}
If $\chi$ is constant on $\mathcal{F}_C$, then $\mathcal{F}_{A}^{e}$ and $\mathcal{F}_D$ form a connected bipartite graph on which $\chi$ uses only two colors;
hence $\chi$ is constant on each of $\mathcal{F}_{A}^{e}$ and $\mathcal{F}_D$.
The case where $\chi$ is constant on $\mathcal{F}_D$ is similar.
This proves the claim.
\end{claimproof}
\medskip

\begin{1claim4}
If $\chi$ is constant on $\mathcal{F}_{A}^{e}$, then it is near-constant on $\mathcal{F}_{C}$ and $\mathcal{F}_{D}$.  If $\chi$ is near-constant on one of $\mathcal{F}_C$ and $\mathcal{F}_D$, then it is also near-constant on the other, together using only two colors on $\mathcal{F}_C \cup \mathcal{F}_D$, and is constant on $\mathcal{F}_{A}^{e}$, using the third color.
\end{1claim4}
\medskip

\begin{claimproof}
To prove the claim, we first observe that if $\chi$ is constant on $\mathcal{F}_{A}^{e}$, then $\mathcal{F}_{C}^{y}$ and $\mathcal{F}_{D}^{z}$ form a connected bipartite graph on which $\chi$ uses only two colors.
Hence, $\chi$ is constant on each of $\mathcal{F}_{C}^{y}$ and $\mathcal{F}_{D}^{z}$ (and, similarly, $\mathcal{F}_{C}^{z}$ and $\mathcal{F}_{D}^{y}$).
It follows immediately that $\chi$ is near-constant on $\mathcal{F}_{C}$ and $\mathcal{F}_{D}$.

Suppose now that $\chi$ is near-constant on $\mathcal{F}_C$. By Claim~3, we may assume that $\chi$ is not constant on ${\mathcal F}_C$, so it uses different colors on $\mathcal{F}_{C}^{y}$ and $\mathcal{F}_{D}^{z}$. Then there is only one color left for the whole of $\mathcal{F}_{A}^{e}$ (noting from Claim 2 that every vertex of $\mathcal{F}_{A}^{e}$ is adjacent to some vertex in each of $\mathcal{F}_{C}^{y}$ and $\mathcal{F}_{C}^{z}$).
So $\chi$ is constant on $\mathcal{F}_{A}^{e}$, and hence near-constant on $\mathcal{F}_D$.
Again, the case where $\chi$ is near-constant on $\mathcal{F}_D$ is similar.
This completes the proof of the claim.
\end{claimproof}

\medskip

Our next goal is to count 3-colorings of $B(X')$. All such colorings are partitioned into three classes.

\medskip

Class 1: Colorings that are constant on $\mathcal{F}_C$ for some $C \in V(B(X)) \setminus \phi_X(e)$.
By Corollary \ref{cor:necessary 1/2 components}, $B(X) - \phi_{X}(e)$ is connected.
This fact, repeatedly combined with Claim 3, implies that such a coloring $\chi$ is constant on $\mathcal{F}_{A}^{e}$ for every $A \in \phi_{X}(e)$.
Note that every coloring of $B(X)$ gives rise to a coloring of $B(X')$ in the obvious way: If $\mathcal{D} \in B^2(X)$ is a color class in a coloring of $B(X)$, we let $\mathcal{D}' = \bigcup_{D\in \mathcal{D}} \mathcal{F}_D \in B^2(X')$. This correspondence yields a bijection between colorings of $B(X)$ and the colorings of $B(X')$ that are of Class 1.
In particular, the number of colorings of Class 1 is equal to the number of triangles of $X$, since $X$ is reflexive. Of course, this is equal to the number of vertices of $H$.
\medskip

Class 2: Colorings that are constant on $\mathcal{F}_{A}^{e}$ for some $A \in \phi_X(e)$, but not on $\mathcal{F}_{C}$ for any $C \in V(B(X)) \setminus \phi_X(e) $.
In this case, Claim 4 implies that any such coloring is near-constant on $\mathcal{F}_C$ for each $C \in V(B(X)) \setminus \phi_X(e)$.
Since $B(X)  - \phi_X(e)$ is connected, the same two colors are used for every $\mathcal{F}_C$ and the color of each $\mathcal{F}_{C}^{y}$ and $\mathcal{F}_{C}^{z}$ are determined by the choice of any one of them.
We can therefore construct exactly one coloring of $B(X')$ in this way.
\medskip

Class 3: Colorings that are non-constant on each $\mathcal{F}_{A}^{e}$, $A \in \phi_X(e)$.
In this case, we will show that any such coloring is completely determined by its restriction to an arbitrarily chosen triangle of clusters $\mathcal{F}_{A}^{e}, \mathcal{F}_{C}^{y}, \mathcal{F}_{D}^{z}$.
The vertices of these clusters can be colored according to any coloring of $B(Y)$, except for the coloring in which the clusters are themselves color classes.
Therefore, the number of colorings covered by this case is one less than the number of triangles in $Y$.

Suppose now a coloring $\chi$ of $B(X')$ of class 3 and consider its restriction to $\mathcal{F}_{A}^{e} \cup \mathcal{F}_{C}^{y} \cup \mathcal{F}_{D}^{z}$. Our aim is to prove that this restriction determines $\chi$.
To see this, we employ an inductive argument.
First observe, by Lemma~\ref{lem:Determined}, that $\chi$ is determined on $\mathcal{F}_{D}^{y}$ and $\mathcal{F}_{C}^{z}$, since these clusters form a connected bipartite graph.
Moreover, by Claim 4, neither $\mathcal{F}_{C}$ nor $\mathcal{F}_{D}$ may be near-constant.
Similarly, $\chi$ is determined on any clusters adjacent to $\mathcal{F}_{A}^{e}$.
Again, by Lemma~\ref{lem:Determined}, noting that $\chi$ is not near-constant on $\mathcal{F}_{C}$, $\chi$ is also determined on $\mathcal{F}_{B}^{e}$ for any $B \in \phi_X(e)$ which is adjacent to $C$, since the bipartite graph between $\mathcal{F}_{B}^{e}$ and $\mathcal{F}_F$ is connected, where $F$ is the third vertex of the triangle in $B(X)$ containing $B$ and $C$.
Similarly to the argument for $\mathcal{F}_{A}^{e}$, $\chi$ is also determined on $\mathcal{F}_{F}$.
Repeating this argument, since $B(X)$ is connected, we conclude that we have uniquely determined the whole of the coloring $\chi$.

By the preceding cases, in total, the number of colorings of $B(X')$, and therefore the number of triangles in $B^2(X')$,
is equal to the sum of the number of triangles of $X$ and $Y$.
Since $X'$ also has this many triangles, we conclude by Lemma \ref{lem:edge-reflexive-countC3} that the injection $\phi_{X'} : X' \rightarrow B^2(X')$ is in fact an isomorphism.
\end{proof}

As a direct corollary of Lemma \ref{lem:gluinghalfedges} we obtain a theorem of Fisk \cite{Fiskbook}.

\begin{corollary}[Fisk \cite{Fiskbook}]\label{cor:Fisktrees}
Every cubic tree is edge-reflexive.
\end{corollary}

The lemma also enables us to restrict ourselves to 2-edge-connected cubic graphs.

\begin{corollary}\label{cor:blocks}
Suppose that all graphs obtained from a connected cubic graph $G$ by cutting all cutedges of $G$ are edge-reflexive. Then $G$ is edge-reflexive.
\end{corollary}

\subsection{2-Connected outerplanar graphs}

Corollary \ref{cor:blocks} shows that in order to prove Theorem~\ref{thm: Main}, it suffices to prove the following.

\begin{lemma}\label{lem:partmain}
Every 2-connected cubic triangle-free outerplanar graph is edge-reflexive.
\end{lemma}

It is well-known (and easy to see) that any 2-connected, triangle-free, cubic, outerplanar graph $G$ can be constructed from a cubic 4-cycle by repeatedly applying the following two operations:
\medskip

1. \emph{Adding a 4-cycle}:
Given an edge \begin{math} e = v_1v_2 \end{math} in a cubic graph $H$, incident with half-edges $e_1$ and $e_2$ (respectively), as well as the full edges $f_1$ and $f_2$, add two new vertices $v_3, v_4$ and form a 4-cycle $v_1v_2v_3v_4$, where $e_1$ joins $v_1$ and $v_4$, $e_2$ joins $v_2$ and $v_3$, and $e_3$ joins $v_3$ and $v_4$.
Finally, add half-edges $e_{v_3}$ and $e_{v_4}$ incident with $v_3$ and $v_4$, respectively.
\medskip

2. \emph{Subdividing an edge}:
Given an edge \begin{math} e = v_1v_2 \end{math} in a cubic graph $H$, where $v_1$ and $v_2$ are incident with half-edges $e_1$ and $e_2$ (respectively), as well as with full edges $f_1$ and $f_2$, subdivide the edge $e$ into two edges $e'$ and $e''$ by inserting a new vertex $v$.
Then add a half-edge $g$ incident with $v$, in order to form a new cubic graph.
\medskip

The two operations with the corresponding notation that will be used when speaking about them are depicted in Figure \ref{fig:operations12}.

\begin{figure}[htb]
     \centering
     {\includegraphics[width=0.8\textwidth]{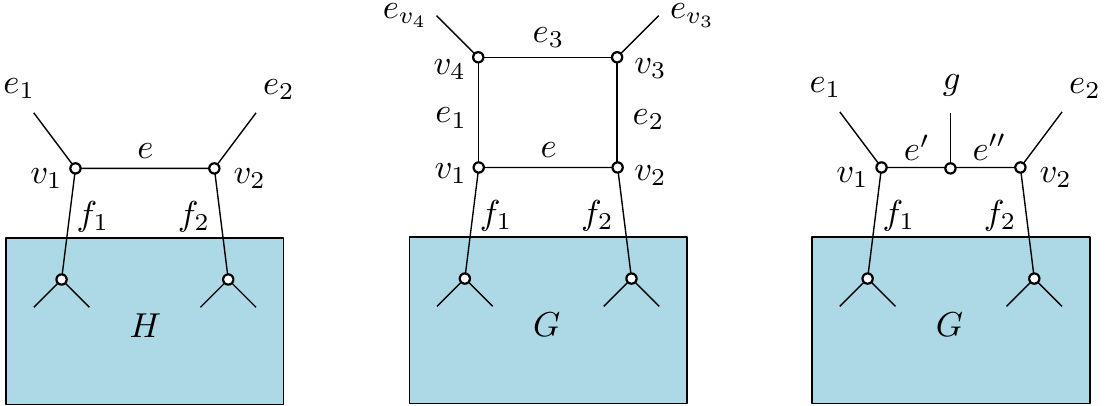}}
     \caption{Adding a 4-cycle and subdividing an edge}
     \label{fig:operations12}
\end{figure}

In order to prove Lemma \ref{lem:partmain}, it suffices to show that these two operations preserve reflexivity.

\begin{lemma}\label{lem:op1}
Operation 1 preserves edge-reflexivity.
\end{lemma}

\begin{proof}
Let $G$ be the graph obtained from $H$ by adding a 4-cycle, let $X = L(G)$, $Y = L(H)$, and assume that $Y$ is reflexive.
We have to show that the homomorphism $\phi_{X} : X \rightarrow B^{2}(X)$ is a bijection.
As before, we will establish injectivity and then count the number of colorings of $B(X)$.
\medskip
\begin{2claim}
The graph $X$ is colorful.
\end{2claim}
\medskip
\begin{claimproof}
We are given that $Y$ is colorful.
Since every 3-coloring of $Y$ extends to $X$, it immediately follows that, for any pair of vertices \begin{math} u, v \in V(X) \setminus \{ e_3, e_{v_3}, e_{v_4} \} \end{math}, there exists a coloring of $X$ in which $u$ and $v$ are colored differently.
Moreover, as there is a coloring of $Y$ which colors $f_1$ and $e_2$ differently, there is a coloring of $X$ in which $e_1$ and $e_2$ have the same color.
Now, by performing an $\{ e_3, e_{v_3}, e_{v_4} \}$ Kempe change, if needed, we can ensure that for any \begin{math} u \in V(X) \setminus \{ e_3, e_{v_3}, e_{v_4} \} \end{math} and \begin{math} v \in \{ e_3, e_{v_3}, e_{v_4} \} \end{math}, there exists a coloring in which $u$ and $v$ are colored differently.
Finally, as there exists a coloring in $Y$ which colors $e_1$ and $e_2$ differently, we can extend this coloring to $X$ in order to obtain a coloring in which $e_{v_3}$ and $e_{v_4}$ are colored differently.
This completes the claim.
\end{claimproof}
\medskip

Now, we partition the vertices of $B(X)$ into the seven sets shown in Figure~\ref{fig:part4}, where we denote by $\mathcal{C}_{x_1x_2 \dots x_t}$ the set of color classes of $X$ which contain the vertices \begin{math} x_1, x_2, \dots, x_t \in \{ e,e_1,e_3,e_2, f_1, f_2 \} \end{math}. In other words,
$$
   \mathcal{C}_{x_1x_2 \dots x_t} = \bigcap_{i=1}^t \phi_X(x_i).
$$
Such sets will be referred to as \emph{clusters}. In Figure~\ref{fig:part4} we also have the cluster $\mathcal{C}_{\widehat{e_3}f_1f_2}$, where $\widehat{e_3}$ indicates that the color classes in this cluster \emph{do not} contain $e_3$. Note that $\mathcal{C}_{e_3f_1f_2} \cup \mathcal{C}_{\widehat{e_3}f_1f_2}$ is a partition of $\mathcal{C}_{f_1f_2}$.

The 3-colorings of $X$ fall into three types as indicated by triangles in Figure~\ref{fig:part4}. The subgraph of $B(X)$ consisting of all triangles of type $i\in \{1,2,3\}$ will be denoted by $\mathcal{T}_i$. Note that $\mathcal{T}_i$ contains only those vertices from the corresponding three clusters that appear as color classes in $\mathcal{T}_i$. Thus, $\mathcal{T}_i$ is obtained from the induced subgraph on the three clusters by removing the isolated vertices (which must participate in colorings of the neighboring $\mathcal{T}_j$, but not in $\mathcal{T}_i$).
We also denote the subgraph of $\mathcal{T}_i$ induced by the union of the two adjacent clusters $C_{x_1 \dots x_t}$ and $C_{w_1 \dots w_s}$ by $C_{x_1\dots x_t} C_{w_1 \dots w_s}$.

\medskip

\begin{2claim2}
$\mathcal{T}_2 \cup \mathcal{T}_3$ is isomorphic to $B(Y)$. Under this isomorphism, $\mathcal{C}_{ee_3}$ is mapped onto $\phi_Y(e)$.
\end{2claim2}
\medskip
\begin{claimproof}
Consider the map \begin{math} \psi: \mathcal{T}_2 \cup \mathcal{T}_3 \rightarrow B(Y) \end{math} which takes a color class $C$ of $X$ to its intersection with $V(Y)$.
As every coloring $c$ of $Y$ extends uniquely to a coloring $c'$ of $X$ with $c'(e) = c'(e_3)$, $\psi$ is a bijection.
Moreover, this unique extension also establishes that $\psi$ is a graph homomorphism, since it follows that $ABC$ is a coloring of $\mathcal{T}_2 \cup \mathcal{T}_3$ if and only if $\psi(A) \psi(B) \psi(C)$ is a coloring of $Y$.
Thus, our claim is established.
\end{claimproof}
\medskip

\begin{figure}[htb]
     \centering
     {\includegraphics[width=0.72\textwidth]{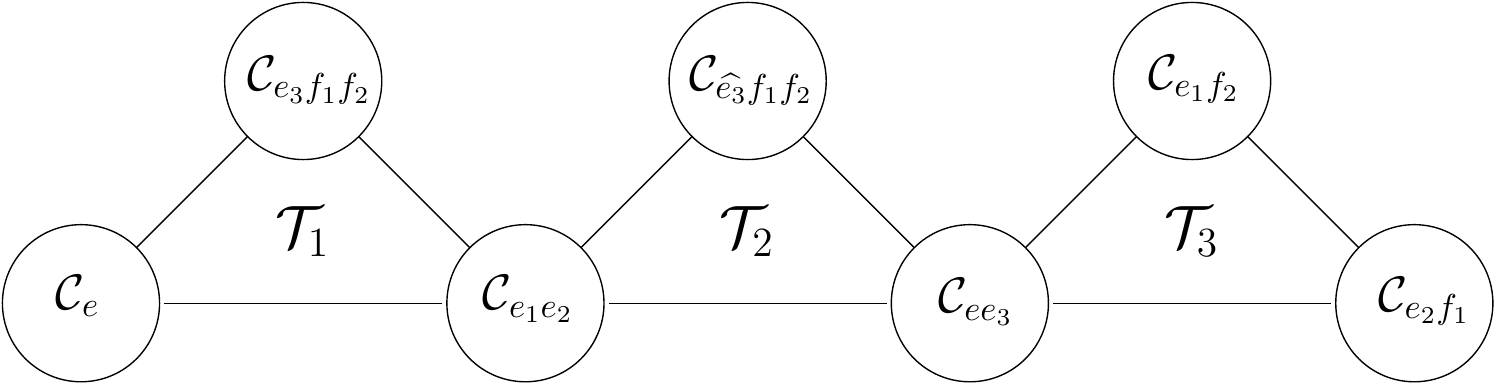}}
     \caption{Partitioning $B(X)$ into clusters}
     \label{fig:part4}
\end{figure}

Since $Y$ is reflexive, Corollary \ref{cor:necessary 1/2 components} shows that $B(Y)-\phi_Y(e)$ has precisely two components. Hence, it follows from Claim 2 that $\mathcal{C}_{e_1 f_2 } \mathcal{C}_{ e_2 f_1 }$ and \begin{math} \mathcal{C}_{ e_1 e_2 } \mathcal{C}_{ \widehat{e_3} f_1 f_2 } \end{math} are bipartite and connected subgraphs of $ \mathcal{T}_2 \cup \mathcal{T}_3$.
\medskip

\begin{2claim3}
$\mathcal{T}_1$ is isomorphic to $\mathcal{T}_2$ and $\mathcal{C}_e \mathcal{C}_{ e_3 f_1 f_2 }$ is isomorphic to $\mathcal{C}_{e e_3} \mathcal{C}_{\widehat{e_3}f_1 f_2}$.
\end{2claim3}
\medskip
\begin{claimproof}
Note that the colorings of $X$ forming $\mathcal{T}_1$ and $\mathcal{T}_2$ have $e_1$ and $e_2$ colored the same.
Consider the map \begin{math} \psi: \mathcal{T}_2 \rightarrow \mathcal{T}_1 \end{math} which takes a color class $C\in V(\mathcal{T}_2)$ to the color class of $X$ in $\mathcal{T}_1$ which results from performing a Kempe change on $\{ e_3, e_{v_4}, e_{v_3} \}$.
Observe that this map is well-defined since $e_{v_3}$ and $e_{v_4}$ are colored the same and hence the colors of $e_1$ and $e_2$ remain the same.
As $\psi$ and $\psi^{-1}$ are both invertible (as Kempe changes are reversible), $\psi$ is a bijection.
Similarly, $\psi$ is a graph homomorphism, since it follows from the existence and reversablility of the Kempe change discussed that $ABC$ is a triangle of $\mathcal{T}_2$ if and only if $\psi(A) \psi(B) \psi(C)$ is a triangle of $\mathcal{T}_1$.
These observations confirm our claim.
\end{claimproof}
\medskip

\begin{2claim4}
$\mathcal{C}_{ \widehat{e_3}f_1 f_2 } \mathcal{C}_{e e_3 }$ is isomorphic to $\mathcal{C}_{ \widehat{e_3}f_1 f_2 } \mathcal{C}_{ e_1 e_2 }$.
\end{2claim4}
\medskip
\begin{claimproof}
Consider the map \begin{math} \psi: \mathcal{T}_2 \rightarrow \mathcal{T}_2 \end{math} which takes a color class $C\in V(\mathcal{T}_2)$ to the color class of $X$ in $\mathcal{T}_2$ which results from performing a Kempe change on $\{ e, e_1, e_2, e_3 \}$.
Observe that this map is well-defined, as such a Kempe change always exists and leaves $e_1$ and $e_2$ the same color, which remains different from the shared color of $e$ and $e_3$.
As Kempe changes are reversible, $\psi$ is a bijection.
Similarly, $\psi$ is a graph homomorphism, since it follows from the existence and reversibility of the Kempe change discussed that $ABC$ is a coloring of $\mathcal{T}_2$ if and only if $\psi(A) \psi(B) \psi(C)$ is a coloring of $\mathcal{T}_2$.
\end{claimproof}
\medskip

Claim 4 implies that $\mathcal{C}_{ e e_3} \mathcal{C}_{\widehat{e_3}f_1 f_2}$ is bipartite and connected. By Claim 3 it follows immediately that $\mathcal{C}_{e} \mathcal{C}_{e_3 f_1 f_2}$ is bipartite and connected.

Now, let $\chi$ be a coloring of $B(X)$. Suppose first that $\chi$ is constant on $\mathcal{C}_{ e e_3 } \cap \mathcal{T}_2$.
We claim that $\chi$ is constant on the whole cluster $\mathcal{C}_{ e e_3 }$.
For a contradiction, suppose that this claim is false.
As $\chi$ is constant on  $\mathcal{C}_{ e e_3 } \cap \mathcal{T}_2$, we know that $\chi$ is constant on all three clusters of $\mathcal{T}_2$.
Now, as $Y$ is reflexive and $\mathcal{T}_2 \cup \mathcal{T}_3$ is isomorphic to $B(Y)$ (by Claim 2), the color class of $\mathcal{T}_2 \cup \mathcal{T}_3 \cong B(Y)$ containing $\mathcal{C}_{ e e_3 } \cap \mathcal{T}_2$ must be of the form $\phi_{Y}(v)$, for some vertex $v \in V(Y)$.
If $v = e$ or $v = e_3$, we are done, so we may assume otherwise.
As $\mathcal{C}_{\widehat{e_3}f_1 f_2}$ and $\mathcal{C}_{e_1 e_2} \cap \mathcal{T}_2$ are both nonempty, the color class containing $\mathcal{C}_{ e e_3 } \cap \mathcal{T}_2$ in $B(Y)$ cannot be any of $\phi_{Y}(f_1)$, $\phi_{Y}(f_2)$, $\phi_{Y}(e_1)$, $\phi_{Y}(e_2)$, $\phi_{Y}(e_{v_4})$ or $\phi_{Y}(e_{v_3})$.
However, for any vertex $v \in V(Y) \setminus \{ e, e_1, e_3, e_2, f_1, f_2, e_{v_4}, e_{v_3} \}$,  by Claim 4, if $\phi_{Y}(v) \cap \mathcal{C}_{e e_3} \cap \mathcal{T}_2 \neq \emptyset$, then $\phi_{Y}(v) \cap \mathcal{C}_{e_1 e_2} \cap \mathcal{T}_2 \neq \emptyset$.
But this contradicts the fact that $\chi$ is constant on $\mathcal{C}_{e e_3} \cap \mathcal{T}_2$, completing our claim.

Consequently, if $\chi$ is constant on $\mathcal{C}_{ e e_3 } \cap \mathcal{T}_2$, then $\chi$ is constant on the whole cluster $\mathcal{C}_{ e e_3 }$, and hence, on all the clusters of $B(X)$.
There are four such colorings of $B(X)$, compared to two such colorings of $B(Y)$.

Let us now consider the case when $\chi$ is not constant on $\mathcal{C}_{ e e_3 } \cap \mathcal{T}_2$.
In this case, by Lemma~\ref{lem:Determined} and Claim 4, $\chi$ is determined and non-constant on $\mathcal{C}_{ e_1 e_2 } \cap \mathcal{T}_2$ and determined on $\mathcal{C}_{ \widehat{e_3}f_1 f_2 }$.
By Claim 3 and Lemma \ref{lem:Determined}, it then follows that $ \chi $ is completely determined on $\mathcal{T}_1$.
Thus, the coloring $\chi$ of $B(X)$ is completely determined by its restriction to $\mathcal{T}_2 \cup \mathcal{T}_3$, which is isomorphic to $B(Y)$.
Each such coloring is determined uniquely by a coloring of $B(Y)$ which is not constant on all the clusters of $B(X)$.
Since two colorings of $B(Y)$ correspond to case 1, we get two fewer colorings in case 2, so we obtain two less colorings than the number of colorings of $B(Y)$.
\medskip

Thus, in total, $B(X)$ has two more colorings than $B(Y)$ had, which is precisely the number of additional triangles in $X$.  Consequently, $X \cong B^{2}(X)$, as required.
\end{proof}

\begin{lemma}\label{lem: subdiv}
Operation 2 preserves edge-reflexivity.
\end{lemma}

\begin{proof}
Again, let $X=L(G)$, $Y=L(H)$, and we assume that $Y$ is reflexive.
As in our previous arguments, we will demonstrate that the homomorphism $\phi_{X} : X \rightarrow B^{2}(X)$ is a bijection by establishing injectivity and then counting the number of colorings of $B(X)$.
We also find it useful to define graphs $H_{e'}$ and $H_{e''}$ as shown in Figure~\ref{fig: restrict}, as well as their line graphs $X_{e'}$ and $X_{e''}$, respectively. Note that $H_{e'}$ and $H_{e''}$ are both isomorphic to $H$, but have some of their edges labelled differently. Through this labeling we can view $E(H_{e'})$ and $E(H_{e''})$ as subsets of $E(G)$.

Figure \ref{fig: restrict} also shows the correspondence of 3-edge-colorings of $H_{e'}$ and $H_{e''}$ with certain 3-edge-colorings of $G$.
Observe that $H_{e'}$ ($H_{e''}$) has precisely the 3-edge colorings of $G$ in which the color of $e'$ is equal to the color of $e_2$ (the color of $e''$ is equal to the color of $e_1$) and $H_{e'} \cong H \cong H_{e''}H$.

\begin{figure}[htb]
	\centering
	\includegraphics[width=\textwidth]{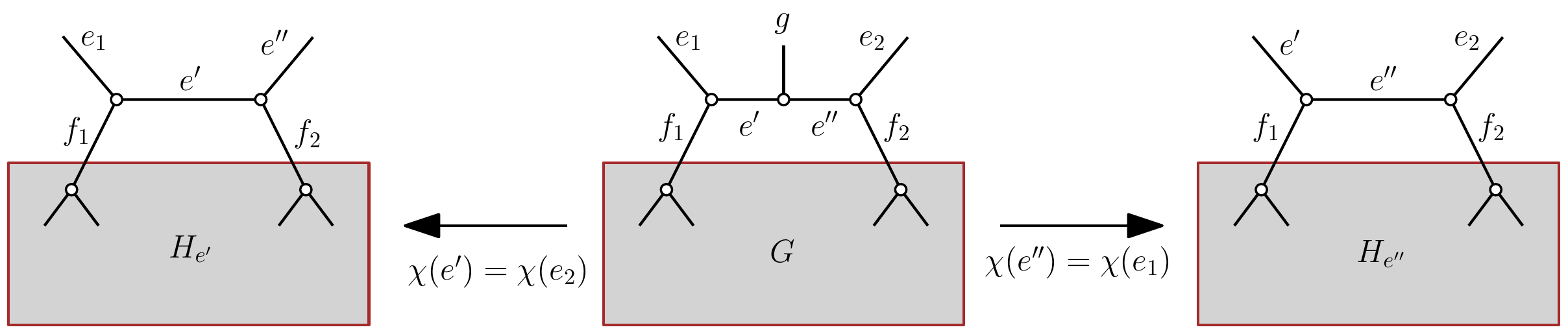} \caption{Two graphs isomorphic to $H$ whose 3-edge-colorings correspond to certain subsets of the 3-edge-colorings of $G$}
\label{fig: restrict}
\end{figure}

\medskip

\begin{2claim}
$G$ is edge-colorful.
\end{2claim}
\medskip
\begin{claimproof}
Let $u, v \in V(X)$.  We examine four cases.
\medskip

Case 1: If $u,v \notin \{e', e'', e_1, e_2, g \}$, then a 3-coloring $c_{e'}$ of $X_{e'}$ with $c_{e'}(u) \neq c_{e'}(v)$ can be extended to a coloring of $X$ in which $c(u) \neq c(v)$.
As $X_{e'}$ is colorful, this establishes case 1.
\medskip

Case 2: $ u,v \in \{  e',  e'', e_1, e_2, g \}$:
As $X_{e'}$ is colorful, there exists a 3-edge coloring of $G$ with $c(f_1) \neq c(f_2)$ and $c(e') = c(e_2)$. Since $c(e')$ is different from $c(f_1)$ and $c(f_2)$, we have that $c(e'') \neq c(e_1)$, $c(e_1) \neq c(e_2)$ and $c(e_2) \neq c(g)$.
By the same argument on $X_{e''}$, we also establish that there is a coloring distinguishing $e', e_2$ and $e_1, g$.
\medskip

Case 3: If $ u \in \{  e',  e'', e_1, e_2 \}$ and $v \notin \{  e',  e'', e_1, e_2, g \}$, then we may always arrange a coloring of $X$ in which $c(f_1) = c(f_2)$, as $X_{e'}$ is colorful.
In the event that $c(u) = c(v)$, we can then perform a Kempe change on the vertices of the set $\{  e',  e'', e_1, e_2 \}$ in order to arrange for $c(u) \neq c(v)$.
This establishes case 3.
\medskip

Case 4: $u = g$ and $v \notin \{  e',  e'', e_1, e_2, g \}$:
In case 2, we showed that there is a 3-edge coloring of $G$ with $c(e') = c(f_2)$.
So, for each vertex $v \notin \{g, e_2\}$ we can arrange for $c(g) \neq c(v)$ by performing a Kempe change on the path $g, e'', e_2$ (if needed).
This completes the claim.
\end{claimproof}
\medskip

Now, as $X$ is colorful (and thus $\phi_{X}: X \rightarrow B^{2}(X)$ is injective), it suffices to establish that $B(X)$ has precisely one more 3-coloring than $B(Y)$.  In order to prove this fact, we will again partition $B(X)$ into \emph{clusters} of the form $\mathcal{C}_{ab} = \phi_X(a)\cap \phi_X(b)$, where $a,b\in V(X)$. We will consider the partition into clusters as depicted in Figure \ref{fig: partsubdiv}.

\begin{figure}[htb]
	\centering
	\includegraphics[width=0.42\textwidth]{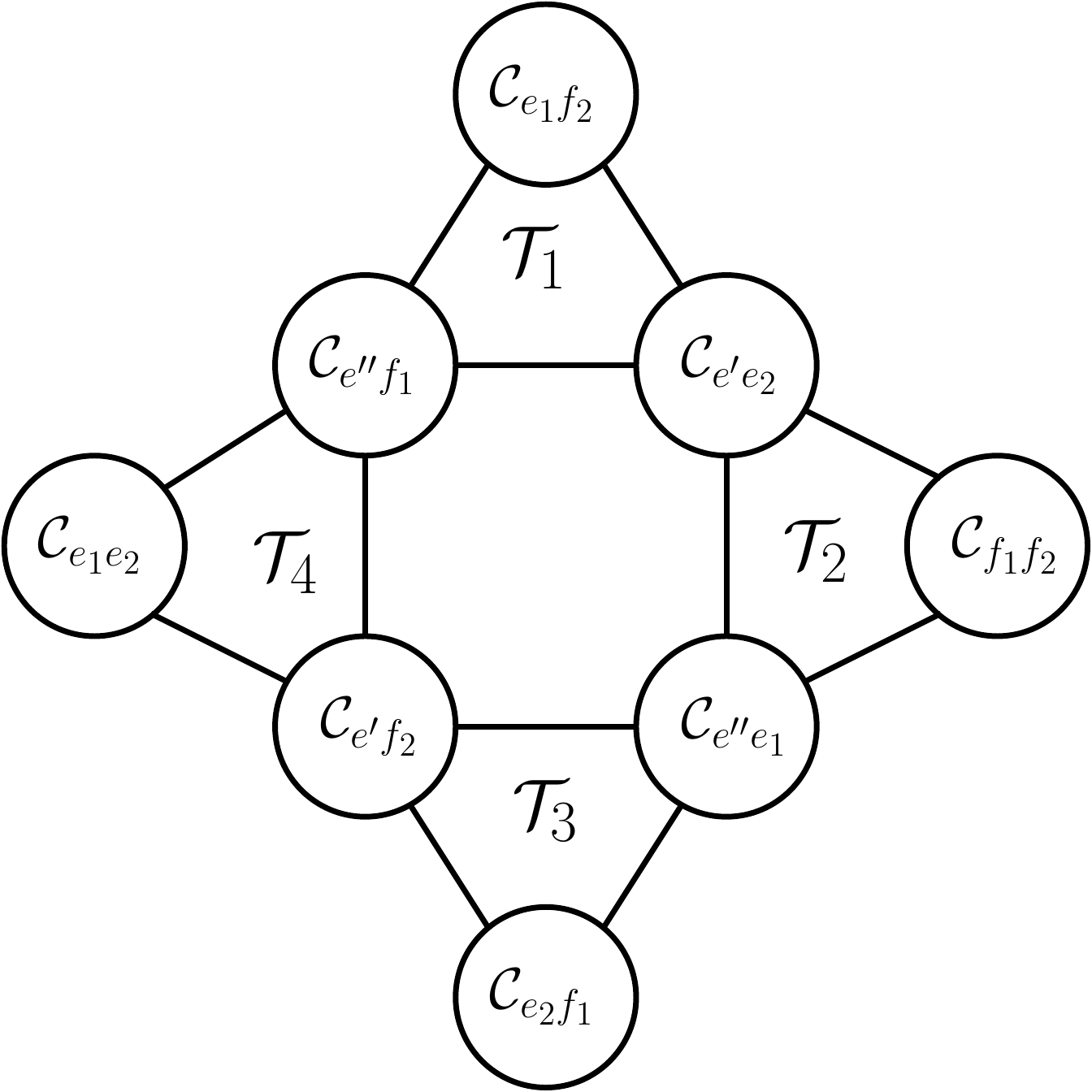}
    \caption{Partitioning $B(X)$ into clusters}
\label{fig: partsubdiv}
\end{figure}

As in the proof of Lemma \ref{lem:op1}, let $\mathcal{T}_{i}$ $(1 \leq i \leq 4)$ be the subgraph of $B(X)$ on all color sets participating in 3-colorings of $X$ whose edges are between the three clusters of $\mathcal{T}_{i}$, as shown in Figure~\ref{fig: partsubdiv}.  We will refer to the subgraph consisting of all edges in $B(X)$ joining the clusters $C_{uv}$ and $\mathcal{C}_{xy}$ as the edge $\mathcal{C}_{uv}\mathcal{C}_{xy}$ of our cluster partition. When $\mathcal{C}_{uv}\mathcal{C}_{xy}$ is nonempty, it is contained in precisely one of the subgraphs $\mathcal{T}_i$, $i\in \{1,2,3,4\}$.
\medskip

\begin{2claim2}
The following subgraphs of $B(X)$ are isomorphic: $\mathcal{T}_1 \cup \mathcal{T}_2 \cong B(X_{e'}) \cong B(Y) \cong B(X_{e''}) \cong \mathcal{T}_2 \cup \mathcal{T}_3$.
\end{2claim2}
\medskip

\begin{claimproof}
That $ B(X_{e'}) \cong B(Y) \cong B(X_{e''}) $ follows immediately from their definitions.
So, it suffices to prove that the maps $\psi_1 :  \mathcal{T}_1 \cup \mathcal{T}_2 \rightarrow B(X_{e'})$ and $\psi_2 :  \mathcal{T}_2 \cup \mathcal{T}_3 \rightarrow B(X_{e''})$, each of which takes a color class $C$ of $X$ to its restriction as indicated in Figure~\ref{fig: restrict}, are isomorphisms.
As these two arguments are identical, we will only establish the claim for the map $\psi_1$.

As every coloring $c$ of $X_{e'}$ extends uniquely to a coloring of $X$ and $\mathcal{T}_1\cup \mathcal{T}_2$ includes the whole $\mathcal{C}_{e'e_2}$, $\psi_1$ is a bijection.
Moreover, this unique extension also establishes that $\psi_1$ is a graph homomorphism, since it follows that a coloring $ABC$ of $X$ is in $\mathcal{T}_1 \cup \mathcal{T}_2$ if and only if $\psi_1 (A) \psi_1 (B) \psi_1 (C)$ is a coloring of $X_{e'}$.
Thus, Claim 2 is resolved.
\end{claimproof}
\medskip

\begin{2claim3}
The edges of the cluster partition in Figure~\ref{fig: partsubdiv} representing $\mathcal{C}_{e_1 f_2} \mathcal{C}_{e'' f_1}$, $\mathcal{C}_{f_1 f_2} \mathcal{C}_{e'' e_1}$, $\mathcal{C}_{e' e_2} \mathcal{C}_{f_1 f_2}$ and $\mathcal{C}_{e' f_2} \mathcal{C}_{e_2 f_1}$ all represent connected, bipartite subgraphs of $B(X)$.
Moreover, $\mathcal{C}_{f_1 f_2} \mathcal{C}_{e'' e_1} \cong \mathcal{C}_{e' e_2} \mathcal{C}_{f_1 f_2}$.
\end{2claim3}
\medskip
\begin{claimproof}
Recall that, by Claim 2, $\mathcal{T}_1 \cup \mathcal{T}_2 \cong B(X_{e'}) \cong B(Y) \cong B(X_{e''}) \cong \mathcal{T}_2 \cup \mathcal{T}_3$.
Since $X_{e'}$ is reflexive and $e'$ is in two triangles of $X_{e'}$, $\phi_{X_{e'}}(e')$ must be in precisely two triangles of $B^{2}(X_{e'})$. Through the isomorphism $\mathcal{T}_1 \cup \mathcal{T}_2 \cong B(X_{e'})$, $\phi_{X_{e'}}(e)$ corresponds to the cluster $\mathcal{C}_{e'e_2}$ in $\mathcal{T}_1 \cup \mathcal{T}_2$.
Therefore, $\mathcal{C}_{e_1 f_2} \mathcal{C}_{e'' f_1}$ and $\mathcal{C}_{f_1 f_2} \mathcal{C}_{e'' e_1}$ are bipartite and connected.
Similarly, as $\phi_{X_{e''}}(e'')$ must be in precisely two triangles of $B^{2}(X_{e''})$, $\mathcal{C}_{e' e_2} \mathcal{C}_{f_1 f_2}$ and $\mathcal{C}_{e' f_2} \mathcal{C}_{e_2 f_1}$ are bipartite and connected.

Now, consider the map \begin{math} \psi: \mathcal{T}_2 \rightarrow \mathcal{T}_2 \end{math}, which takes a color class $C\in V(\mathcal{T}_2)$ to the color class $C'$ which results from performing a Kempe change on $\{ e_1, e', e'', e_2 \}$.
Observe that this map is well-defined, as such a Kempe change always exists and leaves $e_1$ and $e''$ the same color, which remains different from the shared color of $e'$ and $e_2$.
Moreover, as Kempe changes are reversible, $\psi$ is a bijection, and $ABC$ is a triangle in $\mathcal{T}_2$ if and only if $\psi(A) \psi(B) \psi(C)$ is a triangle in $\mathcal{T}_2$.
Consequently, $\psi$ is a graph isomorphism.
It remains to note that $\psi$ maps $\mathcal{C}_{f_1 f_2} \mathcal{C}_{e'' e_1}$ onto $\mathcal{C}_{e' e_2} \mathcal{C}_{f_1 f_2}$. This establishes Claim 3.
\end{claimproof}
\medskip

We will now discuss the structure of 3-colorings of $B(X)$. If we consider our cluster partition as an 8-vertex graph, each 3-coloring of that graph determines a 3-coloring of $B(X)$ in which each cluster is monochromatic (the coloring is \emph{constant on the cluster}). There are other 3-colorings of $B(X)$. To understand them, we first show that each such coloring is determined by its restriction to certain subgraphs of $B(X)$.

\medskip

\begin{2claim4}
For every 3-coloring of $B(X)$, its restriction to $\mathcal{T}_1 \cup \mathcal{T}_3$ determines the coloring on $\mathcal{T}_4$.
\end{2claim4}
\medskip
\begin{claimproof}
Observe that deleting either $e'$ and $f_2$ or $e''$ and $f_1$ separates $X$ into two components, one of which only contains vertices in $\{e_1, e_2, e', e'', g \}$.
Consequently, we can obtain from any coloring of $X$ represented by a triangle in $\mathcal{T}_4$ a triangle in $\mathcal{T}_1$ through a Kempe change on $\{e_1,e',g\}$ and a triangle in $\mathcal{T}_3$ through a Kempe change on $\{e_2,e'',g\}$. This shows that $\mathcal{C}_{e''f_1} = \mathcal{T}_1 \cap \mathcal{T}_3$ and $\mathcal{C}_{e'f_2} = \mathcal{T}_3 \cap \mathcal{T}_4$.
Thus, the coloring of $\mathcal{C}_{e'' f_1}$ is determined by the coloring of $\mathcal{C}_{e'' f_1} \cap \mathcal{T}_1$, and the coloring of $\mathcal{C}_{e^{\prime } f_2}$ is completely determined by the coloring of $\mathcal{C}_{e' f_2} \cap \mathcal{T}_3$.
Now, every vertex in $\mathcal{C}_{e_1e_2}$ is in a triangle in $\mathcal{T}_4$ and hence also its color is determined.
\end{claimproof}
\medskip

\begin{2claim5}
If a coloring $\chi$ of $X$ is constant on $\mathcal{C}_{e'' e_1} \cap \mathcal{T}_2$, then $\chi$ is constant on the whole cluster $\mathcal{C}_{e'' e_1}$.  If $\chi$ is constant on $\mathcal{C}_{e' e_2} \cap \mathcal{T}_2$, then it is constant on $\mathcal{C}_{e' e_2}$.
\end{2claim5}
\medskip

\begin{claimproof}
For a contradiction, suppose that this claim is false.
If $\chi$ is constant on  $\mathcal{C}_{e'' e_1} \cap \mathcal{T}_2$, then (by Claim 3) it is constant on each of $\mathcal{C}_{e' e_2} \cap \mathcal{T}_2$ and $\mathcal{C}_{f_1 f_2}$.
Now, as $\mathcal{T}_2 \cup \mathcal{T}_3 \cong B(X_{e''})$ is reflexive, the color class containing $\mathcal{C}_{e'' e_1} \cap \mathcal{T}_2$ in $B(X_{e''})$ must be of the form $\phi_{X_{e''}}(v)$, for some vertex $v \in V(X_{e''})$.
If $v = e''$, we are done, so we may assume otherwise.
As $\mathcal{C}_{f_1 f_2}$ and $\mathcal{C}_{e' e_2} \cap \mathcal{T}_2$ are both nonempty, the color class containing $\mathcal{C}_{ e'' e_1 } \cap \mathcal{T}_2$ in $B(X_{e''})$ cannot be any of $\phi_{X_{e''}}(f_1)$, $\phi_{X_{e''}}(f_2)$, $\phi_{X_{e''}}(e'')$,  $\phi_{X_{e''}}(e_2)$ or $\phi_{X_{e''}}(e')$.
However, when the vertex $v \in V(X_{e''}) \setminus \{e_2, e', e'', f_1, f_2\}$, then $\phi_{X_{e''}}(v) \cap \mathcal{C}_{{e' e_2}} \cap \mathcal{T}_2 \neq \emptyset$ by Claim 3.
But this contradicts the fact that $\chi$ is constant on $\mathcal{C}_{e'' e_1} \cap \mathcal{T}_2$, completing our claim.
The proof of the second statement is the same.
\end{claimproof}
\medskip

\begin{2claim6}
The subgraphs of $B(X)$ corresponding to $\mathcal{T}_4$ and $\mathcal{T}_3$ are isomorphic, and $\mathcal{C}_{e'' f_1} \mathcal{C}_{e' f_2}$ is bipartite and connected.
\end{2claim6}
\medskip

\begin{claimproof}
As in the proof of Claim 3, we consider the map $\psi: \mathcal{T}_4 \rightarrow \mathcal{T}_3$ induced on $B(X)$ by performing a Kempe change in $X$ on $ \{ g, e'', e_2 \}$.
This map is a bijection between $\mathcal{T}_4$ and $\mathcal{T}_3$.
Moreover, $ABC$ is a triangle in $\mathcal{T}_4$ if and only if $\psi(A) \psi(B) \psi(C)$ is a triangle in $\mathcal{T}_3$, so this is indeed a graph isomorphism.
\end{claimproof}
\medskip

\begin{2claim7}
If a 3-coloring $\chi$ of $B(X)$ is non-constant on $\mathcal{C}_{e'' e_1} \cap \mathcal{T}_2$, then the restriction of $\chi$ on $(\mathcal{C}_{e'' e_1} \cap \mathcal{T}_2) \cup (\mathcal{C}_{e' f_2}\cap \mathcal{T}_3)$ determines $\chi$ on the whole of $\mathcal{T}_3$.
\end{2claim7}
\medskip

\begin{claimproof}
Suppose that $A \in (\mathcal{C}_{e'' e_1} \cap \mathcal{T}_3) \setminus (\mathcal{C}_{e'' e_1} \cap \mathcal{T}_2)$.
As $\mathcal{T}_2 \cup \mathcal{T}_3 \cong B(X_{e''})$ and $B(X_{e''}) \setminus \phi_{X_{e''}}(e')$ is bipartite and connected (by Corollary~\ref{cor:necessary 1/2 components}), there exists a path in $\mathcal{C}_{e'' e_1} \mathcal{C}_{e_2 f_1}$ from $A$ to some vertex $D \in \mathcal{C}_{e'' e_1} \cap \mathcal{T}_2$.
Since $D \in \mathcal{C}_{e'' e_1} \cap \mathcal{T}_2$, its color is determined.
Since the coloring on $\mathcal{C}_{e' f_2} \cap \mathcal{T}_3$ is determined, we conclude that $\chi$ is determined on the whole path from $D$ to $A$. Thus, the color of $A$ is determined.
Now, as $A$ was chosen arbitrarily, it follows that the coloring is determined on $\mathcal{C}_{e'' e_1} \cap \mathcal{T}_3$, from which it follows that $\chi$ is determined on all of $\mathcal{T}_3$.
This proves the claim.
\end{claimproof}
\medskip

\begin{2claim8}
If a 3-coloring $\chi$ of $B(X)$ is non-constant on $\mathcal{C}_{e'' e_1} \cap \mathcal{T}_2$, then the restriction of $\chi$ on $(\mathcal{C}_{e'' e_1} \cap \mathcal{T}_2) \cup (\mathcal{C}_{e'' f_1} \cap \mathcal{T}_4)$ determines $\chi$ on the whole of $\mathcal{T}_4$.
\end{2claim8}
\medskip

\begin{claimproof}
Firstly, let $A \in \mathcal{C}_{e' f_2}$ be a vertex of $B(X)$ in a connected component $K$ of $\mathcal{C}_{e' f_2} \mathcal{C}_{e'' e_1} $.
Since $e_2$ is a half-edge in $X_{e''}$ and $X_{e''}$ is reflexive, $B(X_{e''}) \setminus \phi_{X_{e''}}(e_2)$ is connected (by Corollary~\ref{cor:necessary 1/2 components}).
Thus, there exists a path from $A$ to some vertex $D \in \mathcal{C}_{e'' e_1} \cap \mathcal{T}_2$, which is also in the connected component $K$.
Now, the color class of $D$ is of the form $\phi_{X_{e'}}(x)$, for some $x \in V(X_{e'})$ (because $\mathcal{T}_1 \cup \mathcal{T}_2$ is isomorphic to $X_{e'}$).
The vertex $D$ is also adjacent to some vertex $D' \in \mathcal{C}_{e' f_2}$, which is in turn adjacent (by the aforementioned isomorphism between $\mathcal{T}_3$ and $\mathcal{T}_4$) to a vertex $D'' \in \mathcal{C}_{e'' f_1}$ whose color is determined.

If $D$ and $D''$ are colored with the same color, then they both belong to $\phi_{X_{e'}}(x)$.
Therefore, $x \in D$ and $x \in D''$.
However, by Claim 6, unless $ x \in \{g, e'', e_2 \}$, $ x \in D$ implies that $x \notin D''$.
Meanwhile, if $x \in \{e'', e_2 \}$, then it easily follows that $\chi$ is constant on $\mathcal{C}_{e'' e_1} \cap \mathcal{T}_2$. Finally, $x$ cannot be equal to $g$.
So, in each case, we have a contradiction.
Thus, $D$ and $D''$ must be colored differently.

Since $D$ and $D''$ are colored differently, the color of $D'$ is determined.
Consequently, each connected component $K$ of $\mathcal{C}_{e' f_2} \mathcal{C}_{e'' e_1}$ must contain some vertex $D' \in \mathcal{C}_{e' f_2}$ whose color is determined.

Now, we make an argument similar to that in the proof of Claim 7.
Suppose that $A \in \mathcal{C}_{e' f_2}$.
By using the isomorphism between $\mathcal{T}_3$ and $\mathcal{T}_4$, we see that, in the connected component $K'$ of $\mathcal{C}_{e' f_2} \mathcal{C}_{e_1 e_2}$ containing $A$, there is a path from $A$ to some vertex $D' \in \mathcal{C}_{e' f_2}$ whose color is determined.
Since the color of $D'$ and the colors of vertices in $\mathcal{C}_{e'' f_1}\cap \mathcal{T}_4$ are determined, the neighbor of $D'$ on this path has its color determined.
By iterating this argument, we conclude that all the vertices on this path have their color determined.
Now, as $A$ was chosen arbitrarily, it follows that the coloring is determined on $\mathcal{C}_{e' f_2}$, from which it follows that the coloring on all of $\mathcal{T}_4$ is determined.
This proves the claim.
\end{claimproof}
\medskip

Now, consider an arbitrary 3-coloring $\chi$ of $B(X)$ and its
restriction $\chi'$ on $\mathcal{T}_1 \cup \mathcal{T}_2 \cong B(X_{e'})$.
We consider two cases.

Firstly, if $\chi$ is constant on $\mathcal{C}_{e'' e_1} \cap \mathcal{T}_2$, then by Claim 3, $\chi$ is constant on each cluster of $\mathcal{T}_2$.
Thus, applying Claim 5 and then Claim 3 again, we observe that $\chi$ is constant on all the clusters of $\mathcal{T}_1$, $\mathcal{T}_2$ and $\mathcal{T}_3$.
It then follows from Claim 4 that $\chi$ is constant on every cluster of $B(X)$.
There are three such colorings, compared to two colorings $\chi'$ of $\mathcal{T}_1 \cup \mathcal{T}_2$ which have this form.

Secondly, if $\chi$ is not constant on $\mathcal{C}_{e'' e_1} \cap \mathcal{T}_2$, we need to show that $\chi $ is uniquely determined on all of $B(X)$ by its restriction $\chi'$ to $\mathcal{T}_1 \cup \mathcal{T}_2$.
So, we apply Claim 8.
The 3-coloring $\chi$ of $B(X)$ is determined and non-constant on $\mathcal{C}_{e'' e_1} \cap \mathcal{T}_2$, and determined on $\mathcal{C}_{e'' f_1} $ by the 3-coloring $\chi'$ of $\mathcal{T}_1 \cup \mathcal{T}_2$.
Thus, the coloring is determined on $\mathcal{T}_4$.
Consequently, the 3-coloring is determined and non-constant on $\mathcal{C}_{e'' e_1} \cap \mathcal{T}_2$, and determined on $\mathcal{C}_{e' f_2} $, so by Claim 7, it is determined on $\mathcal{T}_3$.

Hence, all 3-colorings of $\mathcal{T}_1 \cup \mathcal{T}_2 \cong B(Y)$ uniquely extend to $B(X)$ (except in the case when $\mathcal{C}_{e'' f_1}$ and $\mathcal{C}_{e'' e_1}$ are colored identically, in which case the coloring extends in two ways to $B(X)$).  This shows that $B(X)$ has precisely one more 3-coloring than $B(Y)$, as required.
\end{proof}

\begin{proof}[Proof of Lemma \ref{lem:partmain} and of Theorem~\ref{thm: Main}]
By Observation \ref{obs:triangle-free}, $G$ cannot be edge-reflexive if it contains a triangle. So let us assume that $G$ is triangle-free.
By applying Lemma~\ref{lem:op1} and Lemma~\ref{lem: subdiv} repeatedly, we can construct any 2-connected, triangle-free, cubic, outerplanar graph from a cubic 4-cycle (which is edge-reflexive).
This establishes Lemma~\ref{lem:partmain}.
Then, taking this result together with Corollary~\ref{cor:blocks}, Theorem~\ref{thm: Main} follows immediately.
\end{proof}

\section{Subdivisions and reflexive theta graphs}

Outerplanar triangle-free cubic graphs have at least four half-edges. It is therefore a natural question whether there are some cubic graphs with less than four half-edges that are edge-reflexive. If there is just one half-edge, then the graph is not 3-edge-colorable and if there are two half-edges, they receive the same color in every edge-coloring, so such a graph is not edge-colorful. Of course, the most interesting class is when no half-edges are present. Let us observe that $K_{3,3}$ is an edge-reflexive graph (see Figure~\ref{fig: BLK33}) without half-edges.

\begin{figure}[H]
    \begin{minipage}{0.28\textwidth}
        \centering
        \includegraphics[width=0.4\textwidth]{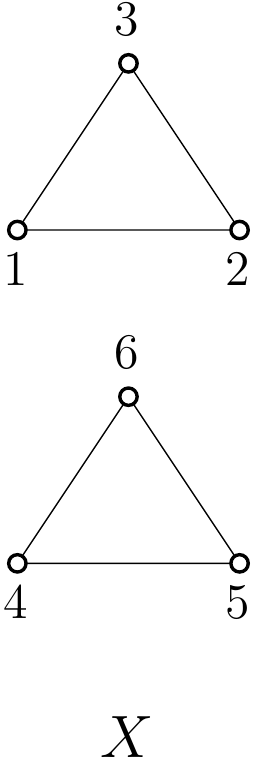}
    \end{minipage}%%%
{\hfill\color{black}\vrule\hfill}%
    \begin{minipage}{0.39\textwidth}
        \centering
        \includegraphics[width=0.95\textwidth]{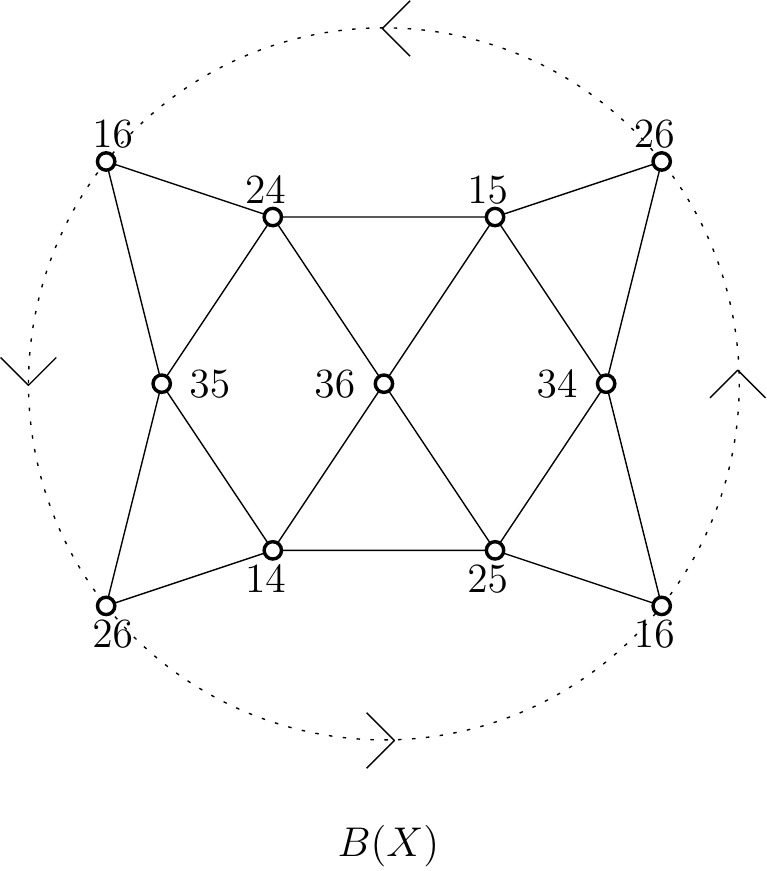}
    \end{minipage}%%
{\hfill\color{black}\vrule\hfill}%
    \begin{minipage}{0.3\textwidth}
        \centering
        \includegraphics[width=0.75\textwidth]{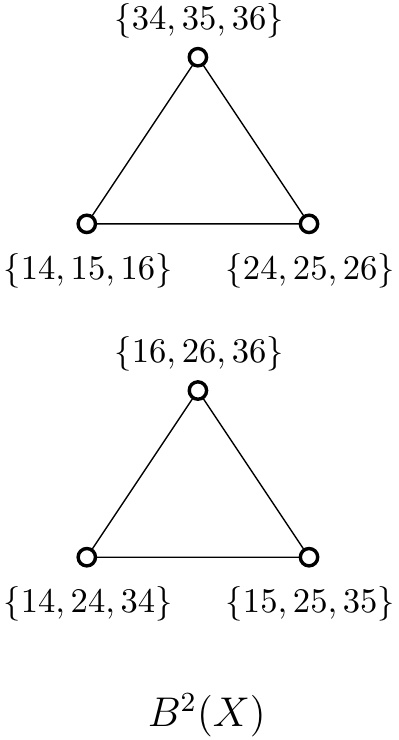}
    \end{minipage}%%

	\caption{The disjoint union of two triangles $X$, its 3-coloring complex $B(X)$ (drawn in the projective plane) which is isomorphic to $L(K_{3,3})$, and the graph $B^{2}(X)$. This shows that $K_{3,3}$ is edge-reflexive.}
	\label{fig: BLK33}
\end{figure}

But there are other classes of reflexive cubic graphs without half-edges. Let us start with the prisms.
The \emph{$n$-prism} $\Pi_n$ is the cubic graph of order $2n$ which is obtained by taking the Cartesian product of an $n$cycle and $K_2$. If $n$ is odd, the $n$-prism is not edge-colorful: in any 3-edge-coloring each pair of the corresponding cycle edges in the Cartesian product is colored the same. This implies that $B^2(L(\Pi_n))$ is isomorphic to $B^2(L(C_n'))$, where $C_n'$ is the cubic $n$-cycle with half-edges. However, prisms of even length are different.

\begin{theorem} \label{thm: evencircladder}
For every even $n\ge4$, the $n$-prism $\Pi_n$ is edge-reflexive.
\end{theorem}

\begin{proof}
Let $A_0$ and $B_0$ be the perfect matchings of the first $n$-cycle $S_0$ in $\Pi_n$ and let $A_1,B_1$ be the corresponding perfect matchings in the second $n$-cycle $S_1$. Also, let $M$ be the perfect matching in $\Pi_n$ consisting of all edges joining the two cycles.

Since $n$ is even, there are two 3-edge-colorings of $\Pi_n$ containing $M$ as a color class: $\{A_0\cup A_1,B_0\cup B_1,M\}$ and $\{A_0\cup B_1,B_0\cup A_1,M\}$. The latter one is called the \emph{mixed coloring}.
Further, any other 3-edge-coloring $\{A,B,C\}$ of the cubic $n$-cycle $S_0\cup M$ (where the edges in $M$ are treated as half-edges) that does not contain the whole $M$ as a color class extends uniquely to a 3-edge-colouring of $\Pi_n$ by adding in each color class all edges in $S_1$ that are copies of the edges of $S_0$ in the color class. This shows that $B(L(\Pi_n))$ is isomorphic to $B(L(C_n'))$ with one added triangle corresponding to the mixed coloring. That triangle shares the color class $M$ with the rest of the coloring complex.

It is easy to see that $\Pi_n$ is edge-colorful. Thus it suffices to show, by Lemma \ref{lem:edge-reflexive-countC3}, that $B(L(\Pi_n))$ has precisely $2n$ 3-colorings. Since $C_n'$ is edge-reflexive (by Theorem \ref{thm: Main}), $B(L(C_n'))$ has precisely $n$ 3-colorings. Each such 3-coloring (on the corresponding subgraph of $B(L(\Pi_n))$ extends in two ways to the whole $B(L(\Pi_n))$ since we have two ways to color the vertices of $B(L(\Pi_n))$ corresponding to color classes $A_0\cup B_1$ and $B_0\cup A_1$ of the mixed coloring.
\end{proof}

In Lemma~\ref{lem: subdiv} we showed that, under certain circumstances, the graph $G'$, which we obtain from an edge-reflexive cubic graph $G$ by subdividing an edge $e$ of $G$, is edge-reflexive.
However, there exist edge-reflexive graphs $G$ where, regardless of how many times we subdivide one of its edges, the result will never be edge-reflexive. One such example is $K_{3,3}$.
In fact, there is a more general family.

\begin{prop} \label{prop:K33}
Suppose that $G$ is an edge-reflexive cubic graph without half-edges.  Then no graph $H$ which results from subdividing a single edge of $G$ $k$ times $(k \geq 1)$ is edge-reflexive.
\end{prop}

In the proof we will employ the well-known Parity Lemma.

\begin{lemma} [Parity Lemma] \label{thm: ParityLemma}
Suppose that a cubic graph $G$ is edge-colored.  Let $n_1$, $n_2$ and $n_3$ be the number of half-edges of $G$ in each of the three color classes. Then $n_1$, $n_2$ and $n_3$ are congruent modulo 2.
\end{lemma}

\begin{proof}
The number of half-edges in a color class is equal to the number $n$ of vertices in $G$, minus twice the number of full edges in the same color class.  Thus, $n \equiv n_1 \equiv  n_2 \equiv n_3 \pmod 2$.
\end{proof}

\begin{proof}[Proof of Proposition~\ref{prop:K33}]
Let the graph $H$ be obtained from $G$ by subdividing $e=uv$ $k\ge1$ times.
The Parity Lemma applied to the graph $H'$ obtained from $H-e$ by adding two half-edges shows that $H$ is not 3-edge-colorable when $k=1$ and that it is not edge-colorful if $k\ge2$ since in every 3-edge-coloring of $H'$, the half-edges are colored the same.
\end{proof}

Proposition~\ref{prop:K33} shows that subdividing a single edge in $K_{3,3}$ yields a graph that is not edge-reflexive.
Of course, this is still not the full story, as some subdivisions of $K_{3,3}$ are edge-reflexive.
For example, if we subdivide each edge of $K_{3,3}$ once, the resulting graph is edge-reflexive.
At this time, we do not fully understand the relation between subdividing edges and edge-reflexivity.
However, we can still use Lemma~\ref{lem: subdiv} to help identify additional infinite families of edge-reflexive graphs.
For example, this lemma is instrumental in proving our next result.

We construct the \emph{cubic theta graph} $T_{k,l,m}$  ($k,l,m \geq 1$) as follows.
Begin with three paths of lengths $k$, $l$ and $m$, respectively.
Label their vertices $u_{0}, u_1, \dots, u_{k}$, $v_{0}, v_1, \dots, v_{l}$ and $w_{0}, w_1, \dots, w_{m}$.
Then identify the vertices $u_{0}$, $v_{0}$ and $w_{0}$, as well as the vertices $u_{k}$, $v_{l}$ and $w_{m}$.
Finally, add half-edges to make the graph cubic.
Observe that $T_{k,l,m} \cong T_{l,k,m} \cong  T_{k,m,l}$, and hence we may assume that $k \leq l \leq m$.

A number of small theta graphs are not edge-reflexive.
In particular, $T_{1,1,m}$ is not edge-reflexive for any $ m \geq 1$ by Proposition~\ref{prop:K33}.
The graph $T_{1,2,m}$ is not edge-reflexive for any $ m \geq 1$, since it contains a triangle.
Additionally, using a computer, we found that $T_{2,2,2}$, $T_{2,2,3}$, $T_{2,2,4}$, $T_{2,3,3}$, $T_{2,3,4}$ and $T_{3,3,3}$ are not edge-reflexive.
However, all other theta graphs are edge-reflexive.

\begin{theorem} \label{thm:outp} The cubic theta graphs $T_{1,1,m}$, $T_{1,2,m}$ $(m \geq 1)$ $T_{2,2,2}$, $T_{2,2,3}$, $T_{2,2,4}$, $T_{2,3,3}$, $T_{2,3,4}$ and $T_{3,3,3}$ are not edge-reflexive.  All other cubic theta graphs are edge-reflexive.
\end{theorem}
\begin{proof}

As mentioned above, the graphs listed in the statement of the theorem are not edge-reflexive.
To show that all other cubic theta graphs are edge-reflexive, we have verified by using computer
that $T_{1,3,3}$, $T_{2,2,5}$, $T_{2,3,5}$, $T_{2,4,4}$ and $T_{3,3,4}$ are edge-reflexive.
Since any other cubic theta graph can be obtained from one of these by subdividing edges, Lemma \ref{lem: subdiv} implies that they are all edge-reflexive.
\end{proof}

\section{Problems}

As we saw in the previous section, many non-outerplanar graphs exist which are edge-reflexive.
A theta graph can have a large number of vertices which do not appear on the outer face of any drawing, while $K_{3,3}$ is a non-planar graph.
Nonetheless, in both of these cases (with the exception of a few small theta graphs) we obtain edge-reflexive graphs.

Another class of potentially edge-reflexive cubic graphs we consider particularly interesting are the fusenes (also known as \emph{hexagonal graphs}).
We say that $G$ is a \emph{fusene} if $G$ is a 2-connected plane graph, in which every interior face is a hexagon, all vertices of $G$ have degree three (after adding half-edges) and only vertices on the boundary of the outer face are permitted to be incident with half-edges.

\begin{question}
Do there exist any fusenes that are not edge-reflexive?
\end{question}

We are not aware of any examples, and have confirmed through a lengthy computation that none exist with nine or fewer hexagonal faces. Given any edge-reflexive fusene graph, we obtain an infinite family of edge-reflexive fusenes by using the operation of adding a 4-cycle followed by two subdivisions. However, not all fusenes are obtained this way.

We have also uncovered a generalization of the theta graphs, which may yield another infinite family of edge-reflexive graphs.
We will call a graph \emph{theta ladder} $TL(l,m,n)$ if it is constructed as follows.
Begin with two 6-cycles $a_1a_2a_3a_4a_5a_6$ and $b_1b_2b_3b_4b_5b_6$.
Now, join the edge $a_1a_2$ to the edge $b_1b_2$ with a ladder of length $k$.
Similarly, the edges $a_3a_4$ and $b_3b_4$ are connected by a ladder of length $l$, while $a_5a_6$ and $b_5b_6$ are connected by a ladder of length $m$.
An example given in Figure~\ref{fig: thetaladders} is the theta ladder $TL(3,3,3)$.

\begin{figure}[H]
    \centering
    \includegraphics[width=0.3\textwidth]{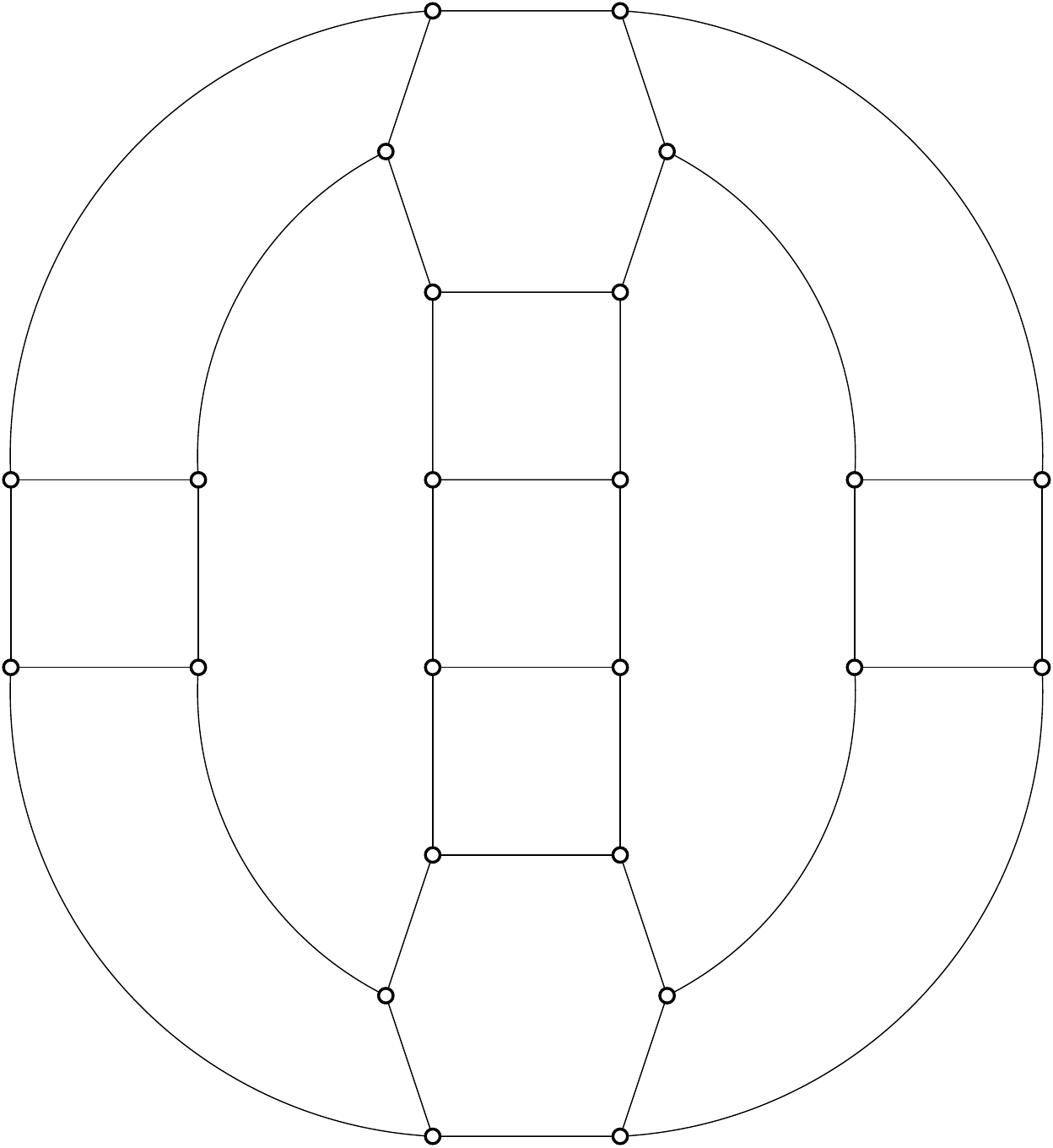}
	\caption{The theta ladder $TL(3,3,3)$ constructed by attaching three ladders to a pair of 6-cycles.  Many graphs of this form are edge-reflexive.}
	\label{fig: thetaladders}
\end{figure}

We have confirmed, using a computer, that $TL(1,1,1)$, $TL(1,3,3)$, $TL(1,3,5)$, $TL(1,3,7)$, $TL(1,3,9)$, $TL(3,3,3)$, $TL(1,5,5)$, $TL(1,5,7)$, $TL(3,3,5)$, $TL(3,5,5)$ and $TL(3,3,7)$ are edge-reflexive. For those $TL(l,m,n)$ where at least one parameter is even and $l+m+n\le 13$ we found out that they are not edge-reflexive, and the same holds for $TL(1,1,3)$ (the only odd-odd-odd exception).
Based on this evidence, we ask the following question.

\begin{question}\label{q:TL}
{\rm (a)} Do there exist any theta ladder graphs $TL(l,m,n)$, where $l,m,n \geq 3$ are all odd, that are not edge-reflexive?

{\rm (b)} Do there exist any theta ladder graphs $TL(l,m,n)$, where $l$ is even, that are edge-reflexive?
\end{question}

The graph $TL(l,m,n)$ is not edge-colorful if one parameter is even and the other two are odd. (We leave the proof of this fact as an exercise.) Thus, in this case $TL(l,m,n)$ is not reflexive. We were not able to establish a similar result for other cases of Question \ref{q:TL}(b).

\bibliographystyle{plain}
\bibliography{references}
\end{document}